\documentclass{article}
\usepackage{amsmath,amsthm}
\usepackage{amssymb,latexsym}
\usepackage{enumerate}

\newtheorem{thm}{Theorem}[section]
\newtheorem{cor}[thm]{Corollary}
\newtheorem{lem}[thm]{Lemma}

\theoremstyle{definition}

\numberwithin{equation}{section}

\begin{document}




\title{Nonlinear mappings preserving at least one eigenvalue}

\author{Constantin Costara\footnote{Corresponding author} \\
Faculty of Mathematics and Informatics \\
Ovidius University\\
Mamaia Blvd. 124, Constan\c{t}a, Romania\\
E-mail: cdcostara@univ-ovidius.ro
\and 
Du\v{s}an Repov\v{s}\\
Faculty of Mathematics and Physics, and Faculty of Education \\
University of Ljubljana\\
P.O.Box 2964, Ljubljana 1001, Slovenia\\
E-mail: dusan.repovs@guest.arnes.si}

\date{}

\maketitle


\renewcommand{\thefootnote}{}



\renewcommand{\thefootnote}{\arabic{footnote}}
\setcounter{footnote}{0}


\begin{abstract}
We prove that if $F$ is a Lipschitz map from the set of all complex $n\times n$ matrices into itself with $F(0)=0$ such that given any $x$ and $y$ we have that $%
F\left( x\right) -F\left( y\right) $ and $x-y$ have at least one common
eigenvalue, then either $F\left( x\right) =uxu^{-1}$ or $F\left( x\right)
=ux^{t}u^{-1}$ for all $x$, for some invertible $n\times n$ matrix $u$. We
arrive at the same conclusion by supposing $F$ to be of class $\mathcal{C}%
^{1}$ on a domain in $\mathcal{M}_{n}$ containing the null matrix, instead
of Lipschitz. We also prove that if $F$ is of class $\mathcal{C}^{1}$ on a
domain containing the null matrix satisfying $F(0)=0$ and $\rho (F\left( x\right) -F\left(
y\right) )=\rho (x-y)$ for all $x$ and $y$, where $\rho \left( \cdot \right) 
$ denotes the spectral radius, then there exists $\gamma \in \mathbb{C}$ of
modulus one such that either $\gamma ^{-1}F$ or $\gamma ^{-1}\overline{F}$
is of the above form, where $\overline{F}$ is the (complex) conjugate of $F$.
\end{abstract}

\section{Introduction and statement of results}

Linear preserver problems deal with the question of characterizing those linear transformations on an algebra which leave invariant a certain subset, function or relation defined on the underlying algebra. The study of such transformations began with Frobenius in 1897, who characterized the linear maps on matrix algebras preserving the determinant. In 1949, Dieudonn\'{e} characterized the invertible linear maps preserving the set of singular matrices. In 1959, Marcus and
Moyls \cite[Theorem 6]{MM} proved that if $T:\mathcal{M}_{n}\rightarrow \mathcal{M}_{n}$ is linear (with respect to complex scalars) and an eigenvalue
preserver, then there exists an invertible $u\in \mathcal{M}_{n}$ such that
either 
\begin{equation}
T\left( x\right) =uxu^{-1}\qquad (x\in \mathcal{M}_{n})\quad \mbox{or}
\qquad T\left( x\right) =ux^{t}u^{-1}\quad (x\in \mathcal{M}_{n}).
\label{1}
\end{equation}%
(For a fixed integer $n\geq 1$, we denote by $\mathcal{M}_{n}$ the space
of all complex $n\times n$ matrices and  $x^{t}$ stands for the transpose of $x \in \mathcal{M}_{n}$.) By a density argument, one
can easily see that we arrive at the same conclusion by supposing $\sigma \left( T\left( x\right) \right)
=\sigma \left( x\right) $ for all $x\in \mathcal{M}_{n}$, where $\sigma \left( x\right) $ stands for the spectrum of $x$, that is the set of all its
eigenvalues without counting multiplicities. 

One of the best known results in the theory of linear preservers is the Gleason--Kahane--\.{Z}elazko theorem (cf. e.g. \cite{Ze}), which asserts that a unital linear functional $f$ defined on a (complex, unital) Banach algebra $\mathcal{A}$ is multiplicative if it preserves invertibility. For unital linear functionals, the assumption of preserving invertibility is easily seen to be equivalent to the condition that $f(x)$ belongs to the spectrum of $x$ for every $x \in \mathcal{A}$. Without assuming linearity, one needs to impose stronger preservation properties on $f$ in order to arrive at the same conclusion. Kowalski and Slodkowski proved in \cite{KK} that every functional $f$ on a Banach algebra $\mathcal{A}$ (no linearity assumed on $f$) with $f(0)=0$ such that the difference of the value of any two elements is contained in the spectrum of the difference of those two elements, is linear and multiplicative. Thus, we may replace the linearity assumption and the spectrum-preserving property of $f$ by a single weaker assumption and still arrive at the same conclusion.

In view of this result, it is quite natural to relax the hypothesis of linearity and try to find a criteria of Kowalski--Slodkowski type for maps defined on matrix spaces to be linear morphisms or antimorphisms. For example, Mr\v{c}un proved that the result of Marcus and Moyls \cite[Theorem 6]{MM}
also holds by supposing $T$ to be only $\mathbb{R}$-linear (that
is, additive and homogeneous with respect to scalars from the real field).

\begin{lem}
\label{l1}\cite[Lemma 3]{Mr} Let $T:\mathcal{M}_{n}\rightarrow \mathcal{M}%
_{n}$ be an $\mathbb{R}$-linear mapping such that $T\left( x\right) $ and $x$
have the same spectrum for all $x\in \mathcal{M}_{n}$. Then $T$ is $\mathbb{C%
}$-linear (and therefore of the form (\ref{1})).
\end{lem}

Mr\v{c}un then used the assertion of Lemma \ref{l1} and ideas from \cite{KK}  to solve the following
non-linear preserver problem on matrix spaces.

\begin{thm}
\label{th1}\cite[Theorem 1]{Mr} Let $F:\mathcal{M}_{n}\rightarrow \mathcal{M}%
_{n}$ be a Lipschitz mapping with $F\left( 0\right) =0$ such that 
\begin{equation}
\sigma \left( F\left( x\right) -F\left( y\right) \right) \subseteq \sigma
\left( x-y\right) \qquad \left( x,y\in \mathcal{M}_{n}\right) .  \label{2}
\end{equation}
Then $T$ is of the form (\ref{1}).
\end{thm}

The most difficult part of the proof of \cite[Theorem 1]{Mr} was to show that
under the assumptions of Theorem \ref{th1} the map $F$ is not far from
being \emph{holomorphic} on $\mathcal{M}_{n}$ (see the proof of \cite[Lemma 4%
]{Mr}), and this heavily depends on the statement of Lemma \ref{l1}. The
following is a generalization of \cite[Lemma 3]{Mr}.

\begin{lem}
\label{l1a} Let $T:\mathcal{M}_{n}\rightarrow \mathcal{M}_{n}$ be an $%
\mathbb{R} $-linear mapping such that $T\left( x\right) $ and $x$ always have at least one common eigenvalue. Then $T$ is $\mathbb{C}$-linear.
\end{lem}

With the same basic idea as the one from the proof of \cite[Theorem 1]{Mr},
but with the use of Lemma \ref{l1a} instead of Lemma \ref{l1}, we obtain the
following generalization of Theorem \ref{th1}.

\begin{thm}
\label{th2} Let $F:\mathcal{M}_{n}\rightarrow \mathcal{M}_{n}$ be a
Lipschitz mapping with $F\left( 0\right) =0$ such that 
\begin{equation}
\sigma \left( F\left( x\right) -F\left( y\right) \right) \bigcap \sigma
\left( x-y\right) \neq \varnothing \qquad \left( x,y\in \mathcal{M}%
_{n}\right) .  \label{3}
\end{equation}%
Then $F$ is of the form (\ref{1}).
\end{thm}

The idea to impose certain spectral properties on $F$ in order to make it
holomorphic also works in the case of $\mathcal{C}^{1}$-functions. Motivated
by considerations related to the study of biholomorphic maps of the open
spectral unit ball of $\mathcal{M}_{n}$, Baribeau and Ransford proved that
if $U$ and $V$ are open subsets of $\mathcal{M}_{n}$ and $F:U\rightarrow V$
is a bijective function of class $\mathcal{C}^{1}$ such that $\sigma
(F(x))=\sigma \left( x\right) $ for every $x\in U$, then $F(x)$ is conjugate
to $x$ for all $x\in U$ \cite[Corollary 1.2]{BR}. For example, this
statement holds when $F$ is a bijective \emph{holomorphic}
spectrum-preserving map. When $F$ is supposed to be holomorphic, we may ask $%
F$ to preserve only one eigenvalue, in order to obtain the same result: it
was proved in \cite[Corollary 1.2]{CR} that if $U\subseteq \mathcal{M}_{n}$
is a domain containing the null matrix, $V\subseteq \mathcal{M}_{n}$ is open
and $F:U\rightarrow V$ is a bijective holomorphic map such that $F\left(
x\right) $ and $x$ always have at least one common eigenvalue, then $F\left( x\right) $
and $x$ are conjugate for all $x\in U$. We use this to prove a statement analogous to the one of Theorem \ref{th2} in the context of $\mathcal{C}^{1}$
functions.

\begin{thm}
\label{th4} Let $U\subseteq \mathcal{M}_{n}$ be a domain containing the null
matrix and $F:U\rightarrow \mathcal{M}_{n}$ a function of class $\mathcal{C}%
^{1}$ with $F\left( 0\right) =0.$ If 
\begin{equation}
\sigma \left( F\left( x\right) -F\left( y\right) \right) \bigcap \sigma
\left( x-y\right) \neq \varnothing \qquad \left( x,y\in U\right) ,
\label{3_1}
\end{equation}%
then $F$ is of the form (\ref{1}).
\end{thm}

For $x\in \mathcal{M}_{n}$, denote by $\rho \left(x\right) $ its spectral radius. It was proved in \cite[Corollary 1.4]{CR} that if $U\subseteq \mathcal{M}_{n}$
is a domain containing the null matrix, $V\subseteq \mathcal{M}_{n}$ is open
and $F:U\rightarrow V$ is a bijective holomorphic map such that $F\left(
x\right) $ and $x$ always have the same spectral radius, then there exists a
complex number $\gamma $ of modulus one such that $\gamma ^{-1}F\left(
x\right) $ and $x$ are conjugate for all $x\in U$. We use this to prove the
following version of Theorem \ref{th4}.

\begin{thm}
\label{th5} Let $U\subseteq \mathcal{M}_{n}$ be a domain containing the null
matrix and $F:U\rightarrow \mathcal{M}_{n}$ a function of class $\mathcal{C}%
^{1}$ with $F\left( 0\right) =0.$ If 
\begin{equation}
\rho \left( F\left( x\right) -F\left( y\right) \right) =\rho \left(
x-y\right) \qquad \left( x,y\in U\right) ,  \label{3_a}
\end{equation}%
then there exists $\gamma \in \mathbb{C}$ of modulus one such that either $\gamma
^{-1}F$ or $\gamma ^{-1}\overline{F}$ is of the form (\ref{1}). (The map $%
\overline{F}:U\rightarrow \mathcal{M}_{n}$ is defined by $\overline{F}\left(
x\right) =\overline{F\left( x\right) },$ the matrix obtained from $F\left(
x\right) $ by entrywise complex conjugation.)
\end{thm}

\section{Proof of Lemma \protect\ref{l1a}}

We shall need the following lemma.

\begin{lem}
\label{l2} Let $u\in \mathcal{M}_{n}$ be an invertible matrix such that $K:%
\mathcal{M}_{n}\rightarrow \mathcal{M}_{n}$ given by%
\begin{equation}
K\left( x\right) =\frac{x+uxu^{-1}}{2}\qquad \left( x\in \mathcal{M}%
_{n}\right)  \label{4}
\end{equation}%
satisfies $K\left( x^{2}\right) =K\left( x\right) ^{2}$ for all $x$. Then $K$
is the identity on $\mathcal{M}_{n}$.
\end{lem}

\begin{proof} For all $x$ we have 
\begin{eqnarray*}
\frac{x^{2}+ux^{2}u^{-1}}{2} &=&K\left( x^{2}\right) =K\left( x\right) ^{2}
\\
&=&\frac{x^{2}}{4}+\frac{xuxu^{-1}}{4}+\frac{uxu^{-1}x}{4}+\frac{ux^{2}u^{-1}%
}{4},
\end{eqnarray*}%
and therefore 
\[
\left( x-uxu^{-1}\right) ^{2}=0\qquad \left( x\in \mathcal{M}_{n}\right) .
\]
In particular, $\rho \left( x-uxu^{-1}\right) =0$ for every $x$ in $\mathcal{%
M}_{n}$. For $x\mapsto xu$ this gives 
\[
\rho \left( xu-ux\right) =0\qquad \left( x\in \mathcal{M}_{n}\right) .
\]
Using \cite[Theorem 5.2.1]{Au} we conclude that $u$ belongs to the center $%
\mathcal{Z}(\mathcal{M}_{n})$ of $\mathcal{M}_{n}$. Then (\ref{4}) implies that $%
K\left( x\right) =x$ for all $x$.
\end{proof}

\begin{proof}[Proof of Lemma \ref{l1a}] If $n=1$ then $T$ is the identity on $%
\mathbb{C}$. So suppose now that $n\geq 2$. Given $r\in \mathbb{R}$, we have 
$\sigma (T(e^{-ir}x))\cap \sigma (e^{-ir}x)\neq \varnothing $, that is $%
\sigma (e^{ir}T(e^{-ir}x))\cap \sigma (x)\neq \varnothing $. Using the $%
\mathbb{R}$-linearity of $T$, we obtain 
\begin{eqnarray*}
e^{ir}T(e^{-ir}x) &=&(\cos r+i\sin r)(T\left( x\right) \cos r-T\left(
ix\right) \sin r) \\
&=&\frac{T\left( x\right) +T\left( ix\right) /i}{2}+e^{2ir}\frac{T\left(
x\right) -T\left( ix\right) /i}{2}.
\end{eqnarray*}%
Thus 
\begin{equation}
\sigma (R\left( x\right) +\zeta S\left( x\right) )\cap \sigma (x)\neq
\varnothing \qquad (x\in \mathcal{M}_{n},\ \left\vert \zeta \right\vert =1),
\label{5}
\end{equation}%
where we denoted 
\[
R\left( x\right) =\frac{T\left( x\right) +T\left( ix\right) /i}{2}\qquad 
\mbox{and}\qquad S\left( x\right) =\frac{T\left( x\right) -T\left( ix\right)
/i}{2}
\]
for $x\in \mathcal{M}_{n}$. Since $T$ is $\mathbb{R}$-linear, the same holds
for $R$. One can easily check that $R\left( ix\right) =iR\left( x\right) $
for all $x$, and therefore $R$ is $\mathbb{C}$-linear.

For $x\in \mathcal{M}_{n}$, denoting by $P\left( \lambda \right) $ its
characteristic polynomial we use (\ref{5}) to show that $\det P\left( R\left(
x\right) +\zeta S\left( x\right) \right) =0$ for all $\zeta $ of modulus
one. Then the same holds for all complex numbers $\zeta $, and in particular 
$\det P\left( R\left( x\right) \right) =0$. This means that $R\left(
x\right) $ and $x$ always have at least one common eigenvalue. Using now \cite[
Theorem 3]{AA} we obtain that $R$ is of the form (\ref{1}). Then without
loss of generality we may suppose that $R\left( x\right) =x$ for all $x\in 
\mathcal{M}_{n}$. (If $R\left( x\right) =uxu^{-1}$ for all $x\in \mathcal{M}%
_{n}$ then we work with $x\mapsto u^{-1}T\left( x\right) u$ instead of $T$,
and if $R\left( x\right) =ux^{t}u^{-1}$ for all $x\in \mathcal{M}_{n}$ then
we work with $x\mapsto (u^{-1}T\left( x\right) u)^{t}$ instead of $T.$) This means that 
\begin{equation}
T\left( x\right) +T\left( ix\right) /i=2x\qquad \left( x\in \mathcal{M}%
_{n}\right) .  \label{6}
\end{equation}%
Define now the mapping $W:\mathcal{M}_{n}\rightarrow \mathcal{M}_{n}$ by
putting 
\[
W\left( x\right) =T\left( \frac{x+x^{\ast }}{2}\right) +iT\left( \frac{%
x-x^{\ast }}{2i}\right) \qquad \left( x\in \mathcal{M}_{n}\right) .
\]
One can easily check that $W$ is $\mathbb{C}$-linear. Moreover, for an
arbitrary Hermitian matrix $x$ we have $W\left( x\right) =T\left( x\right) $%
, and therefore $\sigma (W\left( x\right) )\cap \sigma (x)\neq \varnothing $%
. Then \cite[Corollary 3]{AA} implies that $W$ is an eigenvalue-preserver,
so by \cite[Theorem 6]{MM} there exists an invertible matrix $u$ such that 
\begin{equation}
W\left( x\right) =uxu^{-1}\qquad (x\in \mathcal{M}_{n})\qquad \mbox{or}%
\quad W\left( x\right) =ux^{t}u^{-1}\quad (x\in \mathcal{M}_{n}).
\label{7}
\end{equation}

Suppose that the first case occurs in (\ref{7}). Then $T\left( x\right)
=uxu^{-1}$ for all $x$ satisfying $x=x^{\ast }$. Given an arbitrary $x\in 
\mathcal{M}_{n}$, writing it as $x=x_{1}+ix_{2}$ with $x_{1},x_{2}$
Hermitian matrices and using (\ref{6}) we get 
\begin{eqnarray*}
T\left( x\right) &=&T\left( x_{1}\right) +T\left( ix_{2}\right)
=ux_{1}u^{-1}+i\left( 2x_{2}-T\left( x_{2}\right) \right) \\
&=&u\left( x_{1}-ix_{2}\right) u^{-1}+2ix_{2},
\end{eqnarray*}%
and therefore 
\begin{equation}
T\left( x\right) =x+\left( ux^{\ast }u^{-1}-x^{\ast }\right) \qquad (x\in 
\mathcal{M}_{n}).  \label{8}
\end{equation}%
Since $\sigma \left(
T\left( \lambda x\right) /\lambda \right) \cap \sigma \left( x\right) \neq
\varnothing $ for all $\lambda \in \mathbb{C}\backslash \{0\}$, it follows that $\sigma (x+\zeta \left( ux^{\ast }u^{-1}-x^{\ast
}\right) )\cap \sigma \left( x\right) \neq \varnothing $ for all $x\in 
\mathcal{M}_{n}$ and $\left\vert \zeta \right\vert =1$ in $\mathbb{C}$. Once
more, for $x\in \mathcal{M}_{n}$ denoting by $P\left( \lambda \right) $ its
characteristic polynomial we have $\det P(x+\zeta \left( ux^{\ast
}u^{-1}-x^{\ast }\right) )=0$ for all $\zeta \in \mathbb{C}$. Thus $\sigma
(x+\zeta \left( ux^{\ast }u^{-1}-x^{\ast }\right) )\cap \sigma \left(
x\right) \neq \varnothing $ for all $x\in \mathcal{M}_{n}$ and $\zeta \in 
\mathbb{C}$. Taking $\zeta =1/2$ we obtain $\sigma (x+\left( ux^{\ast
}u^{-1}-x^{\ast }\right) /2)\cap \sigma \left( x\right) \neq \varnothing $
for all $x\in \mathcal{M}_{n}$. Define $K:\mathcal{M}_{n}\rightarrow 
\mathcal{M}_{n}$ by putting%
\begin{eqnarray*}
K\left( x\right) &=&x+\left( uxu^{-1}-x\right) /2 \\
&=&(x+uxu^{-1})/2.
\end{eqnarray*}%
Then $K$ is $\mathbb{C}$-linear and $\sigma \left( K\left( x\right) \right)
\cap \sigma \left( x\right) \neq \varnothing $ for all $x\in \mathcal{M}_{n}$
satisfying $x=x^{\ast }$. Using once more \cite[Corollary 3]{AA}, we obtain that $K$
is a Jordan morphism. Then Lemma \ref{l2} implies that $K\left( x\right) =x$
for all $x$. Then $uxu^{-1}=x$ for all $x$, and (\ref{8}) gives $T\left(
x\right) =x$ for all $x$. In particular, $T$ is $\mathbb{C}$-linear.

Let us now prove that the second case in (\ref{7}) cannot occur. If $W\left(
x\right) =ux^{t}u^{-1}$ for all $x$, then exactly as in the proof of (\ref{8}%
) we obtain that $T\left( x\right) =x+\left( u\overline{x}u^{-1}-x^{\ast
}\right) $ for all $x$, where $\overline{x}$ is the matrix obtained from $x$
by entrywise conjugation. Once more, $\sigma \left( T\left( \lambda
x\right) /\lambda \right) \cap \sigma \left( x\right) \neq \varnothing $ for
all nonzero $\lambda $ gives $\sigma (x+\left( u\overline{x}u^{-1}-x^{\ast
}\right) /2)\cap \sigma \left( x\right) \neq \varnothing $ for all $x\in 
\mathcal{M}_{n}$. Define $J:\mathcal{M}_{n}\rightarrow \mathcal{M}_{n}$ by
putting
\begin{eqnarray*}
J\left( x\right)  &=&x+\left( ux^{t}u^{-1}-x\right) /2 \\
&=&(x+ux^{t}u^{-1})/2.
\end{eqnarray*}%
Then $J$ is $\mathbb{C}$-linear and $\sigma \left( J\left( x\right) \right)
\cap \sigma \left( x\right) \neq \varnothing $ for all $x\in \mathcal{M}_{n}$
satisfying $x=x^{\ast }$. It follows by \cite[Corollary 3]{AA} that $J$ is
either a morphism or an antimorphism. In both cases we have $\hbox{Tr}%
J\left( xy\right) =\hbox{Tr}(J\left( x\right) J\left( y\right) )$ for all $%
x,y\in \mathcal{M}_{n}$, where $\hbox{Tr}(\cdot )$ denotes the usual trace
on $\mathcal{M}_{n}$. Using the properties of the trace, this gives%
\[
\hbox{Tr}(x\left( 2y-uy^{t}u^{-1}-u^{t}y^{t}(u^{t})^{-1}\right) )=0\qquad
(x,y\in \mathcal{M}_{n}).
\]
It follows that $y=(uy^{t}u^{-1}+u^{t}y^{t}(u^{t})^{-1})/2$ for all $y\in \mathcal{%
M}_{n}$, which can be rewritten as 
\[
u^{-1}z^{t}u=\frac{z+(u^{-1}u^{t})z(u^{-1}u^{t})^{-1}}{2}\qquad (z\in 
\mathcal{M}_{n}).
\]
Since $z\mapsto u^{-1}z^{t}u$ is a Jordan morphism on $\mathcal{M}_{n}$, using
Lemma \ref{l2} we obtain that $u^{-1}z^{t}u=z$ for all $z$ in $\mathcal{M}_{n}$. Then $ab=(u^{-1}a^{t}u)(u^{-1}b^{t}u)=u^{-1}\left( ba\right) ^{t}u=ba$ for all $a,b\in \mathcal{M}_{n}$, and we arrive
at a contradiction.
\end{proof}

We now use Lemma \ref{l1a} to give a new proof of a special case of a
result of Bhatia, \v{S}emrl and Sourour on maps on matrices that preserve
the spectral radius distance \cite[Theorem 1.1]{BSS}.

\begin{cor}
\label{c1} Let $T:\mathcal{M}_{n}\rightarrow \mathcal{M}_{n}$ be an $\mathbb{%
R}$-linear mapping such that $\rho (T\left( x\right) )=\rho (x)$ for all $x$%
. Then either $T$ or $\overline{T}$ is $\mathbb{C}$-linear.
\end{cor}

\begin{proof} 
Even though $T$ is not supposed to be $\mathbb{C}$-linear,
it shares all the basic properties of spectral isometries \cite{MSo}. Let us
first prove that $T$ is bijective. Since $T$ is $\mathbb{R}$-linear and $%
\mathcal{M}_{n}$ is of finite dimension over $\mathbb{R}$, it suffices
to prove that $T$ is injective. Since $T$ is additive, we must prove that $T\left( x\right) =0$ implies $x=0$. This follows from the identity 
\[
\rho (y)=\rho (T(y))=\rho (T(y)+T(x))=\rho (T(y+x))=\rho (y+x),
\]
which is valid for all $y\in \mathcal{M}_{n}$, and from the characterization
of the radical given by \cite[Theorem 5.3.1]{Au}.

We now prove that $T$ sends the center $\mathcal{Z}(\mathcal{M}_{n})$ of $%
\mathcal{M}_{n}$ into itself. By \cite[Theorem 5.2.2]{Au}, we have $x\in 
\mathcal{Z}(\mathcal{M}_{n})$ if and only if there exists $M>0$ such that $%
\rho \left( x+y\right) \leq M(1+\rho \left( y\right) )$ for all $y\in 
\mathcal{M}_{n}$. So if $x=\lambda I_{n}$ for some $\lambda \in \mathbb{C}$,
then for all $y\in \mathcal{M}_{n}$ we have 
\begin{eqnarray*}
\rho \left( T(x)+y\right) &=&\rho \left( T(x)+T(z)\right) =\rho \left(
x+z\right) \\
&\leq &M(1+\rho \left( z\right) )=M(1+\rho \left( y\right) ).
\end{eqnarray*}%
(Since $T$ is bijective, there exists $z$ such that $T\left( z\right) =y$.)
Thus $T(x)\in \mathcal{Z}(\mathcal{M}_{n})$. In particular, $T\left(
I_{n}\right) =\mu I_{n}$ for some $\left\vert \mu \right\vert =1$. So by
multiplying $T$ with $\mu ^{-1}$, we may suppose that $T\left( I_{n}\right)
=I_{n}$. Also, $T\left( iI_{n}\right) =\zeta I_{n}$ for some complex number $\zeta$ with $\left\vert
\zeta \right\vert =1$. Therefore 
\[
\sqrt{r^{2}+s^{2}}=\rho ((s+ri)I_{n}))=\rho (T((s+ri)I_{n}))=\rho ((s+\zeta
r)I_{n}))=\left\vert s+r\zeta \right\vert
\]
for all $r,s\in \mathbb{R}$ implies that either $\zeta =i$ or $\zeta =-i$.
In the first case we have $T\left( iI_{n}\right) =iI_{n}$ and hence, by the $%
\mathbb{R}$-linearity of $T$, we get $T\left( \lambda I_{n}\right) =\lambda
I_{n}$ for all $\lambda \in \mathbb{C}$. In the second case we obtain $%
\overline{T\left( \lambda I_{n}\right) }=\lambda I_{n}$ for all $\lambda \in 
\mathbb{C}$.

We complete the proof by showing, for example, that if $T\left( \lambda
I_{n}\right) =\lambda I_{n}$ for all $\lambda \in \mathbb{C}$ then $T$
preserves the peripheral spectrum. If this is true, then $x$ and $T\left(
x\right) $ always have at least one common eigenvalue and Lemma \ref{l1a} implies that $%
T$ is $\mathbb{C}$-linear. So let $x\in \mathcal{M}_{n}$ and consider $%
\lambda \in \sigma \left( x\right) $ such that $\rho \left( x\right)
=\left\vert \lambda \right\vert $. Then 
\[
\rho \left( T\left( x\right) +\lambda I_{n}\right) =\rho \left( T(x+\lambda
I_{n})\right) =\rho \left( x+\lambda I_{n}\right) =2\left\vert \lambda
\right\vert =2\rho \left( x\right) ,
\]
and hence there exists $\alpha \in \sigma (T\left( x\right) )$ such that $%
\left\vert \alpha +\lambda \right\vert =2\left\vert \lambda \right\vert $.
Since $\left\vert \alpha \right\vert \leq \rho (T(x))=\left\vert \lambda
\right\vert $, it follows that $\alpha =\lambda $. Therefore, the peripheral spectrum
of $x$ lies inside the peripheral spectrum of $T\left( x\right) $. Since $%
T^{-1}$ also preserves the spectral radius and $T^{-1}\left( \lambda
I_{n}\right) =\lambda I_{n}$ for all $\lambda \in \mathbb{C}$, we may
conclude that $x$ and $T\left( x\right) $ always have the same peripheral
spectrum.
\end{proof}

\begin{cor}
\label{c2} Suppose that $T:\mathcal{M}_{n}\rightarrow \mathcal{M}_{n}$ is an 
$\mathbb{R}$-linear mapping which preserves the spectral radius. There
exists then a complex number $\gamma $ of modulus one such that either $%
\gamma ^{-1}T$ or $\gamma ^{-1}\overline{T}$ is of the form (\ref{1}).
\end{cor}

\begin{proof} 
We use Corollary \ref{c1} and we suppose, for example, that $%
T$ is $\mathbb{C}$-linear. As we have seen in the proof of Corollary \ref{c1}%
, since $\rho (T\left( x\right) )=\rho (x)$ for all $x$ we obtain that $T\left(
I_{n}\right) =\gamma I_{n}$ for some $\left\vert \gamma \right\vert =1$.
Thus $\gamma ^{-1}T:\mathcal{M}_{n}\rightarrow \mathcal{M}_{n}$ is a unital
spectral isomorphism, and therefore a Jordan isomorphism \cite[Corollary 5]%
{MSo}. Thus $\gamma ^{-1}T$ is of the form (\ref{1}).
\end{proof}

\section{Proofs for the case of Lipschitz functions}

Let us first recall that the real (respectively, complex) differential of a
function $F:\mathcal{M}_{n}\rightarrow \mathcal{M}_{m}$ at a point $x_{0}\in 
\mathcal{M}_{n}$ is a mapping $\left( DF\right) _{x_{0}}:\mathcal{M}%
_{n}\rightarrow \mathcal{M}_{m}$ which is linear with respect to the real
(respectively, complex) scalars, and such that 
\[
\lim_{x\rightarrow 0}\frac{\left\Vert F\left( x_{0}+x\right) -F\left(
x_{0}\right) -(DF)_{x_{0}}(x)\right\Vert }{\left\Vert x\right\Vert }=0.
\]
(We work with the operator norm on $\mathcal{M}_{n}$.)

We shall use of the following lemmas. The first one is a consequence of a
result of Rademacher \cite[page 50]{Z}.

\begin{lem}
\label{l3} Let $n$ be a positive integer and $f:\mathcal{M}_{n}\rightarrow 
\mathbb{C}$ a Lipschitz function. Then $f$ has real differentials a.e. with
respect to the Lebesgue measure $m$ on $\mathcal{M}_{n}$.
\end{lem}

The second one is the following generalization of a result from \cite[Lemma
3.2]{KK}.

\begin{lem}
\label{l4}\cite[Lemma 4]{Mr} Let $f:\mathcal{M}_{n}\rightarrow \mathbb{C}$
be a Lipschitz function and assume that $f$ has complex differentials a.e.
with respect to the Lebesgue measure on $\mathcal{M}_{n}$. Then for all $%
a,b\in \mathcal{M}_{n}$, the function $f_{a,b}:\mathbb{C}\rightarrow \mathbb{%
C}$ given by 
\[
f_{a,b}\left( z\right) =f\left( a+bz\right) \qquad \left( z\in \mathbb{C}%
\right)
\]
is affine.
\end{lem}

\begin{proof}[Proof of Theorem \ref{th2}] We use the ideas from \cite[the proof
of Theorem 1]{Mr}. Let $F_{s,k}:\mathcal{M}_{n}\rightarrow \mathbb{C}$ be
given by 
\[
F_{s,k}\left( x\right) =(F(x))_{s,k}\qquad (x\in \mathcal{M}_{n}),
\]
for $s,k=1,...,n$. Since $F$ is Lipschitz this implies that all the mappings $%
F_{s,k}$ are Lipschitz. By Lemma \ref{l3}, the function $F_{s,k}$ has real
differentials on $\mathcal{M}_{n}\backslash Z_{s,k}$, with $m(Z_{s,k})=0$.
Denoting $Z=\bigcup_{s,k=1}^{n}Z_{s,k}$, we see that $m(Z)=0$ and real
differential $(DT_{s,k})_{x}$ exists for all $s,k=1,...,n$ and for all $x\in 
\mathcal{M}_{n}\backslash Z$. Then for all $x\in \mathcal{M}_{n}\backslash Z$
we conclude that $F:\mathcal{M}_{n}\rightarrow \mathcal{M}_{n}$ has real
differential $\left( DF\right) _{x}:\mathcal{M}_{n}\rightarrow \mathcal{M}%
_{n}$ at $x$ given by 
\[
\left( DF\right) _{x}\left( u\right) =((DF_{s,k})_{x}(u))_{s,k=1}^{n}\qquad
(u\in \mathcal{M}_{n}).
\]
Given $x\in \mathcal{M}_{n}\backslash Z$, we have by (\ref{3}) that for all strictly positive integer $m$ 
\[
\sigma \left( \frac{F\left( x+u/m\right) -F\left( x\right) }{1/m}\right)
\bigcap \sigma \left( u\right) \neq \varnothing \qquad (u\in \mathcal{M}%
_{n}).
\]
For $m\rightarrow \infty $, using the continuity
properties for the spectrum it follows that 
\[
\sigma (\left( DF\right) _{x}\left( u\right) )\bigcap \sigma \left( u\right)
\neq \varnothing \qquad (x\in \mathcal{M}_{n}\backslash Z;\;u\in \mathcal{M}%
_{n}).
\]
Then Lemma \ref{l1a} shows that $\left( DF\right) _{x}:\mathcal{M}
_{n}\rightarrow \mathcal{M}_{n}$ is $\mathbb{C}$-linear, and hence $
(DF_{s,k})_{x}:\mathcal{M}_{n}\rightarrow \mathbb{C}$ are $\mathbb{C}$
-linear mappings for all $x\in \mathcal{M}_{n}\backslash Z$ and $s,k=1,...,n$. By Lemma \ref{l4}, fixing any $a,b\in \mathcal{M}_{n}$ we have
that $\lambda \mapsto F_{s,k}\left( a+\lambda b\right) $ is affine, and
hence the same must hold for $F_{a,b}:\mathbb{C}\rightarrow \mathcal{M}_{n}$
given by 
\[
F_{a,b}\left( \lambda \right) =F\left( a+\lambda b\right) \qquad \left(
\lambda \in \mathbb{C}\right) .
\]
Thus $F_{a,b}\left( \lambda \right) =\lambda
(F_{a,b}(1)-F_{a,b}(0))+F_{a,b}(0)$ for all $\lambda \in \mathbb{C}$, that
is 
\begin{equation}
F(a+\lambda b)=\lambda (F(a+b)-F(a))+F(a)\qquad (\lambda \in \mathbb{C}%
;\;a,b\in \mathcal{M}_{n}).  \label{9}
\end{equation}%
Taking $a=0$ in (\ref{9}) we obtain that $F(\lambda b)=\lambda F(b)$ for
every $\lambda \in \mathbb{C}$ and for every $b\in \mathcal{M}_{n}$. Taking $%
\lambda =1$ and replacing $b$ by $(c-a)/2$ in (\ref{9}) we get $%
F(c)=F(a+c)-F(a)$, which yields 
\[
F(a+c)=F(a)+F(c)\qquad (a,c\in \mathcal{M}_{n}).
\]
Thus $F:\mathcal{M}_{n}\rightarrow \mathcal{M}_{n}$ is $\mathbb{C}$-linear.

For $y=0$ in (\ref{3}) we have that $\sigma \left( F\left( x\right) \right)
\bigcap \sigma \left( x\right) \neq \varnothing $ for all $x\in \mathcal{M}%
_{n}$. Since $F$ is $\mathbb{C}$-linear, it follows from \cite[Theorem 2]{AA}
that $F $ is indeed of the form (\ref{1}).
\end{proof}

\section{Proofs for the case of $\mathcal{C}^{1}$ functions}

For conjugate matrices $x$ and $y$ in $\mathcal{M}_{n}$, we shall write $%
x\sim y$.

\begin{proof}[Proof of Theorem \ref{th4}] By hypothesis, the function $F$
has real differential $\left( DF\right) _{x}$ at every point $x$ of $U$. The
proof of Theorem \ref{th2} shows that in fact $\left( DF\right) _{x}:%
\mathcal{M}_{n}\rightarrow \mathcal{M}_{n}$ is $\mathbb{C}$-linear, for all $%
x$ in $U$. This implies that $F$ is holomorphic on $U$ \cite[Theorem 1.2]{Ran}. By (\ref{3_1}) we also have that 
\[
\sigma (\left( DF\right) _{x}\left( v\right) )\bigcap \sigma \left( v\right)
\neq \varnothing \qquad (x\in U;\;v\in \mathcal{M}_{n}),
\]
and using \cite[Theorem 2]{AA} we get that $\left( DF\right) _{x}$ is of the
form (\ref{1}), for all $x$ in $U$. In particular $\left( DF\right) _{0}:%
\mathcal{M}_{n}\rightarrow \mathcal{M}_{n}$ is invertible, and by the
Inverse mapping theorem there exists $\delta >0$ such that by denoting $W$
the open ball (with respect to the operator norm) of center $0$ and
radius $\delta $ in $\mathcal{M}_{n}$ we have that $W\subseteq U$, $F\left(
W\right) \subseteq \mathcal{M}_{n}$ is open and $F:W\rightarrow F\left(
W\right) $ is biholomorphic. For $x$ with $\left\Vert x\right\Vert <\delta
/2 $, let us define $\varphi _{x}:\{h:\left\Vert h\right\Vert <\delta
/2\}\rightarrow \mathcal{M}_{n}$ by putting $\varphi _{x}\left( h\right)
=F\left( x+h\right) -F\left( x\right) $. By (\ref{3_1}), we have $\sigma
(\varphi _{x}\left( h\right) )\bigcap \sigma \left( h\right) \neq
\varnothing $ for all $\left\Vert h\right\Vert <\delta /2$. Using \cite[%
Corollary 1.2]{CR} we obtain that $\varphi _{x}\left( h\right) \sim h$ for $%
\left\Vert h\right\Vert <\delta /2$. That is, $F\left( x+h\right) -F\left(
x\right) \sim h$ for all $x,h$ with $\left\Vert x\right\Vert ,\left\Vert
h\right\Vert <\delta /2$.

There exists an invertible $u\in \mathcal{M}_{n}$ such that either $\left(
DF\right) _{0}(v)=uvu^{-1}$ for all $v$, or $\left( DF\right)
_{0}(v)=uv^{t}u^{-1}$ for all $v$. If the first case occurs, we consider $%
G:U\rightarrow \mathcal{M}_{n}$ given by $G\left( x\right) =u^{-1}F\left(
x\right) u$, and in the second case we consider $G:U\rightarrow \mathcal{M}%
_{n}$ given by $G\left( x\right) =(u^{-1}F\left( x\right) u)^{t}$. In both
cases, we obtain a holomorphic function $G$ on $U$ such that $G\left(
0\right) =0$, $G^{\prime }\left( 0\right) $ is the identity on $\mathcal{M}%
_{n}$, and 
\begin{equation}
G\left( x+h\right) -G\left( x\right) \sim h\qquad (\left\Vert x\right\Vert
,\left\Vert h\right\Vert <\delta /2).  \label{10}
\end{equation}%
We finish the proof by showing that $G$ with the above properties must be
the identity on $U$. To see this, first observe that by taking $x=0$ in (\ref%
{10}) we have $G\left( h\right) \sim h$ for all $h$ with $\left\Vert
h\right\Vert <\delta /2$. Then $G\left( h\right) ^{2}\sim h^{2}$ for all
such $h$, and thus $\hbox{Tr}(G\left( h\right) )=\hbox{Tr}(h)$ and $\hbox{Tr}%
(G\left( h\right) ^{2})=\hbox{Tr}(h^{2})$ for $\left\Vert h\right\Vert
<\delta /2$. Also, for $\left\Vert x\right\Vert ,\left\Vert y\right\Vert
<\delta /4$ by (\ref{10})  we have $(G\left( x\right) -G(y))\sim (x-y)$. This gives 
$(G\left( x\right) -G(y))^{2}\sim (x-y)^{2}$, and in particular 
\[
\hbox{Tr}((G\left( x\right) -G(y))^{2})=\hbox{Tr}((x-y)^{2}).
\]
For $\left\Vert x\right\Vert ,\left\Vert y\right\Vert
<\delta /4$, since $\hbox{Tr}(G\left( x\right) ^{2})=\hbox{Tr}(x^{2})$ and $\hbox{Tr}
(G\left( y\right) ^{2})=\hbox{Tr}(y^{2})$ it follows that $\hbox{Tr}(G\left( x\right)
G(y))=\hbox{Tr}(xy)$.  Thus for all $x$ with $\left\Vert x\right\Vert <\delta /4$ and
for all $y$ we have that 
\begin{equation}
\hbox{Tr}(G\left( x\right) G(\lambda y))=\lambda \hbox{Tr}(xy)  \label{11}
\end{equation}%
for all $\lambda $ in a neighborhood of $0\in \mathbb{C}$. Since $G^{\prime
}\left( 0\right) $ is the identity, it follows that $G(\lambda y)=\lambda
y+\sum_{j=2}^{\infty }\lambda ^{j}b_{j}$ for $\lambda $ near $0$, for some $%
(b_{j})_{j\geq 2}\subseteq \mathcal{M}_{n}$. By equating the coefficients of $%
\lambda $ in (\ref{11}) we get 
\[
\hbox{Tr}((G\left( x\right) -x)y)=0\qquad (\left\Vert x\right\Vert <\delta
/4;\;y\in \mathcal{M}_{n}).
\]
Thus $G\left( x\right) =x$ for $\left\Vert x\right\Vert <\delta /4$, and by
the identity principle for holomorphic functions we obtain that $G\left(
x\right) =x$ for all $x\in U$.
\end{proof}

\begin{proof}[Proof of Theorem \ref{th5}] Using the continuity of the spectral
radius, from (\ref{3_a}) we obtain 
\[
\rho (\left( DF\right) _{x}\left( v\right) )=\rho \left( v\right) \qquad
(x\in U;\;v\in \mathcal{M}_{n}).
\]
By Corollary \ref{c1}, given any $x\in U$ we have that either $\left(
DF\right) _{x}$ is $\mathbb{C}$-linear, or $\left( DF\right) _{x}$ is $%
\overline{\mathbb{C}}$-linear. Denoting 
\[
U_{1}=\{x\in U:\left( DF\right) _{x}\mbox{ is }\mathbb{C}\mbox{
-linear}\}
\]
and
\[
U_{2}=\{x\in U:\left( DF\right) _{x}\mbox{ is }\overline{\mathbb{C}}%
\mbox{-linear}\},
\]
we conclude that $U_{1},U_{2}\subseteq U$ are disjoint subsets (by Corollary \ref{c2},
we have that $\left( DF\right) _{x}:\mathcal{M}_{n}\rightarrow \mathcal{M}%
_{n}$ is always invertible, and hence we cannot have $\left( DF\right)
_{x}=0 $ for some $x$ in $U$) with $U_{1}\bigcup U_{2}=U$. Also, since $F$ is
of class $\mathcal{C}^{1}$ it follows that $U_{1},U_{2}\subseteq U$ are
(possible empty) closed subsets. Since $U$ is a domain, one of the $U_{j}$
is empty. So suppose, for example, that $U=U_{1}$. (In the remaining case,
we work with $\overline{F}$ instead of $F$.) Then \cite[Theorem 1.2]{Ran}
shows that $F$ is holomorphic on $U$.

We now proceed as in the proof of Theorem \ref{th4}; we shall use 
\cite[Corollary 1.4]{CR} instead of \cite[Corollary 1.2]{CR}. For $0\in U$,
by Corollary \ref{c2} there exists a complex number $\gamma $ of modulus one
and an invertible $u$ such that either $\gamma ^{-1}\left( DF\right)
_{0}\left( v\right) =uvu^{-1}$ for all $v$, or $\gamma ^{-1}\left( DF\right)
_{0}\left( v\right) =uv^{t}u^{-1}$ for all $v$. Suppose, for example, that
the first case occurs, and define $G:U\rightarrow \mathcal{M}_{n}$ by
putting $G\left( x\right) =\gamma ^{-1}u^{-1}F\left( x\right) u$. Then $G$
is holomorphic, $G\left( 0\right) =0$, $G^{\prime }\left( 0\right) =$the
identity of $\mathcal{M}_{n}$, and (\ref{3_a}) gives 
\begin{equation}
\rho \left( G\left( x\right) -G\left( y\right) \right) =\rho \left(
x-y\right) \qquad \left( x,y\in U\right) .  \label{12}
\end{equation}%
Since $G^{\prime }\left( 0\right) $ is invertible, there exists $\delta >0$
such that $G\left( W\right) \subseteq \mathcal{M}_{n}$ is open and $%
G:W\rightarrow G\left( W\right) $ is biholomorphic, where $W$ is the open ball
with center $0$ and radius $\delta $ in $\mathcal{M}_{n}$. Taking $y=0$ in (%
\ref{12}) we have $\rho \left( G\left( x\right) \right) =\rho \left(
x\right) $ for all $x\in U$. Using \cite[Corollary 1.4]{CR}, we establish the
existence of $\xi \in \mathbb{C}$ of modulus one such that $G\left( x\right)
\sim \xi x$ for $\left\Vert x\right\Vert <\delta $. Then $G\left( \lambda
I_{n}\right) =\xi \lambda I_{n}$ for $\left\vert \lambda \right\vert <\delta 
$, and taking the derivative with respect to $\lambda $ and using that $%
G^{\prime }\left( 0\right) $ is the identity, we get $\xi =1$. Thus 
\begin{equation}
G\left( x\right) \sim x\qquad \left( \left\Vert x\right\Vert <\delta \right)
.  \label{13}
\end{equation}%
For $\left\Vert x\right\Vert <\delta /2$ in $\mathcal{M}_{n}$, define $%
\varphi _{x}:\{h:\left\Vert h\right\Vert <\delta /2\}\rightarrow \mathcal{M}%
_{n}$ by putting $\varphi _{x}\left( h\right) =G\left( x+h\right) -G\left(
x\right) $. By (\ref{12}), we have $\rho (\varphi _{x}\left( h\right) )=\rho
\left( h\right) $ for all $\left\Vert h\right\Vert <\delta /2$. Using \cite[%
Corollary 1.4]{CR} we verify the existence of $\xi _{x}\in \mathbb{C}$ of
modulus one such that $\varphi _{x}\left( h\right) \sim \xi _{x}h$ for $%
\left\Vert h\right\Vert <\delta /2$. That is, $G\left( x+h\right) -G\left(
x\right) \sim \xi _{x}h$ for all $h$ with $\left\Vert h\right\Vert <\delta /2
$. For $h=\left( \delta /4\right) I_{n}$ we get $G\left( x+\left( \delta
/4\right) I_{n}\right) -G\left( x\right) =\xi _{x}\left( \delta /4\right)
I_{n}$. Taking the trace and using (\ref{13}) we obtain that $\xi _{x}=1$.
Therefore, $G\left( x+h\right) -G\left( x\right) \sim h$ for all $x,h$ with $%
\left\Vert x\right\Vert ,\left\Vert h\right\Vert <\delta /2$, and from this,
exactly as in the final part of the proof of Theorem \ref{th4}, we conclude
that $G$ is the identity on $U$.
\end{proof}

\subsection*{Acknowledgements}
This research was supported by the Slovenian Research Agency
grants P1-0292-0101, J1-9643-0101 and J1-2057-0101. We acknowledge the referee for comments and suggestions.

\end{document}